\documentclass[english]{article}
\usepackage[T1]{fontenc}
\usepackage{amsmath}
\usepackage{amsthm}
\usepackage{amssymb}
\usepackage{graphicx}

\makeatletter
\numberwithin{equation}{section}
\numberwithin{figure}{section}
\theoremstyle{plain}
\newtheorem{thm}{\protect\theoremname}
  \theoremstyle{plain}
  \newtheorem{prop}[thm]{\protect\propositionname}
  \theoremstyle{plain}
  \newtheorem*{cor*}{\protect\corollaryname}
  \theoremstyle{plain}
  \newtheorem*{thm*}{\protect\theoremname}
  \theoremstyle{plain}
  \newtheorem*{prop*}{\protect\propositionname}

\makeatother

\usepackage{babel}
  \providecommand{\corollaryname}{Corollary}
  \providecommand{\propositionname}{Proposition}
  \providecommand{\theoremname}{Theorem}
\providecommand{\theoremname}{Theorem}

\begin{document}
\newcommand{\TD}{\mathbb{T}_{d}}
\newcommand{\ZD}{\mathbb{Z}^{d}}

\title{The Mutating Contact Process: Model Introduction and Qualitative
Analysis of Phase Transitions in its Survival}

\author{Idan Alter\thanks{Bar Ilan University, alteridan@gmail.com} \and
Gideon Amir\thanks{Bar Ilan University, amirgi@macs.biu.ac.il}}
\maketitle
\begin{abstract}
We introduce and study the mutating contact process, a variant of
the multitype contact process, where one type mutates at a constant
rate to the other type. We prove that on $\mathbb{Z}$ a single mutant
cannot survive while on $\TD$ there are distinct weak survival and
extinction of a single mutant phases, yet the limiting distribution
concentrates on configurations with no mutants of the first type for
any values of the parameters.
\end{abstract}

\section{Introduction and Results}

In this paper we wish to introduce and study a variant of the multitype
contact process introduced by Neuhauser in \cite{Neuhauser1992}.
In that model, two populations, say mutation strains of a virus, are
in competition for space on a certain graph ( the lattice $\ZD$ or
the regular tree $\TD$ most usually): each member of population $i$
reproduces to its nearest neighbors at rate ${\lambda _i}$( $i=1,2$)
if they are unoccupied and dies at rate 1. 

The variant which we wish to study henceforward is the mutating contact
process: in it each individual of the population reproduces at rate
1 to nearest neighbors if they are unoccupied, dies at rate $\delta$
and mutates to a completely novel strain at rate $\mu$. The scope
of our discussion will be limited to the question of the survival
of a single strain: under what conditions does a single strain overcome
death and mutation pressures to persist in existence on the chosen
graph. To this end it suffices to consider the model with only two
types: 1. A specific strain. 2. All other strains. 

With this in mind, the model description is as follows. The system
is described by an evolving configuration $\xi\in{{\left\{ 0,1,2\right\} }^{S}}$
where $\xi \left( a \right)=0$ if the site $a\in S$ is vacant, $\xi \left( a \right)=1$
if it is occupied by a mutant of the specific strain whose survival
we study and $\xi \left( a \right)=2$ if $a$ is occupied by a mutant
of any other strain. Denote by $n\left( x,\xi ,i \right)$ the number
of neighbors of $x$ that are of type $i$ in the configuration $\xi$
, That is,
\[
\ensuremath{n\left(x,\xi,i\right)=\sum\limits _{\begin{smallmatrix}y\sim x\\
\xi\left(y\right)=i
\end{smallmatrix}}1,\,\,\,i=1,2}.
\]
The mutating contact process ${{\xi}_{t}}$ is the Feller process
on ${{\left\{ 0,1,2\right\} }^{S}}$ which makes local transitions
at site $x\in S$ according to the following rules: 
\begin{enumerate}
\item $i\to0\,\,\,\text{at rate }\delta$, $i=1,2$.
\item $0\to i\,\,\,\text{at rate }n\left(x,\xi,i\right)$, $\,i=1,2$. 
\item $1\to2\,\,\text{at rate }\mu$. 
\end{enumerate}
The multitype contact process introduced by Neuhauser \cite{Neuhauser1992}
follows similar rules of evolution: It is a Feller process defined
on ${{\left\{ 0,1,2\right\} }^{S}}$ evolving at $x\in S$ according
to the rules :
\begin{enumerate}
\item $i\to0\,\,\,\text{at rate 1}$, $i=1,2$.
\item $0\to i\,\,\,\text{at rate }{{\lambda}_{i}}n\left(x,\xi,i\right)$,
$\,i=1,2$.
\end{enumerate}
The variations in our model being the introduction of the third evolution
rule, modelling mutation, the requirement that all strains reproduce
at the same constant rate 1, modelling selective neutrality, and the
introduction of a varying death rate instead of a varying birth rate,
which is more amenable to our analysis as will be seen in what follows.

The one-type contact process (obtained for instance by initializing
the two-type process with only one of the types) has been studied
extensively, see \cite{liggett2013stochastic} for a review on the
main results. However, there are relatively few papers on the multitype
contact process, with \cite{Neuhauser1992,Cox2009,Andjel2010} as
representative examples. Further, a review paper by Durrett \cite{durrett2009coexistence}
summarized the results and open questions in spatial competition models. 

in this paper we wish to study the phases of survival of the mutating
contact process. On $\mathbb{{Z}}$, the situation is straightforward
- a single mutant always dies while the entire process behaves exactly
as a single type contact process. The first of these assertions is
proved in the following theorem. In what follows, we will denote the
set of all vertices occupied by type $i$ in the mutating contact
process by $^{i}{{\xi}_{t}}$, that is, $^{i}{{\xi}_{t}}=\left\{ x\in S|{{\xi}_{t}}\left(x\right)=i\right\} $.
\begin{thm}
\label{thm: no strong survival on Z}for the mutating contact process
on $S=\mathbb{{Z}}$, with any finite initial configuration of 1s,
and any choice of death and mutation rates, $\delta,\mu>0$, no single
mutant can survive: 
\[
P\left(^{1}{{\xi}_{t}}\ne\varnothing,\,\forall t>0\right)=0.
\]
\end{thm}

This result builds upon the results of Andjel et. al. \cite{Andjel2010}
in which it is shown that for the multitype contact process with equal
birth rates for the two strains, initialized with a finite number
of 1s bounded from both sides by an infinite number of 2s, the survival
of the 1s is impossible in the same manner as stated above. In the
proof of Theorem 1 we will in effect reduce our problem to theirs. 

The behaviour of the model on $\TD$ is more involved - while there
are distinct weak and strong survival phases for the one type contact
process, it turns out that there is a weak survival phase for a single
mutant in the mutating contact process as well, though strong survival
is an impossibility, these are the results of the next thereom and
the two propositions that folllow it. 
\begin{thm}
\label{thm:1 dies on TD}For the mutating contact process on $S=\TD$
or $S=\ZD$ with any initial configuration, and any choice of $\delta,\mu>0$,
type 1 mutants do not survive in the limiting distribution, that is,
for all $x\in S$,
\[
P\left({{\xi}_{t}}\left(x\right)=1\right)\to0\text{ as }\ensuremath{t\to\infty.}
\]
\end{thm}

The proof of this theorem follows lines similar to the proof of Theorem
2 in Cox and Schinazi \cite{Cox2009}, which states that for the multitype
contact process on $\TD$ or $\ZD$ at the phase of strong survival
of 2s, there can be no coexistence in the limiting distribution. 

The main idea behind the proof is that for ${{\xi}_{t}}\left(x\right)=1$
to occur, it must have descended from an initial 1 without mutating.
As $\ensuremath{t\to\infty}$ it becomes very likely that the ancestor
of $x$ has mutated along its path, and so ${{\xi}_{t}}\left(x\right)=1$
becomes unlikely. 

We will further show that with small enough $\delta$ and $\mu$,
weak survival of 1s on $\TD$ is a possibility, this we state as a
proposition and its proof follows from observing an embedded supercritical
Galton-Watson tree.
\begin{prop}
\label{prop:GW survival}for the mutating contact process on ,$S=\TD$
with any initial finite configuration such that$\left\{ x:{{\xi}_{0}}\left(x\right)=1\right\} \ne\varnothing$,
a single mutant can survive in the weak sense,
\[
P\left(^{1}{{\xi}_{t}}\ne\varnothing,\,\forall t>0\right)>0,
\]
if $\delta+\mu$ is sufficiently small.
\end{prop}

\begin{figure}
\includegraphics{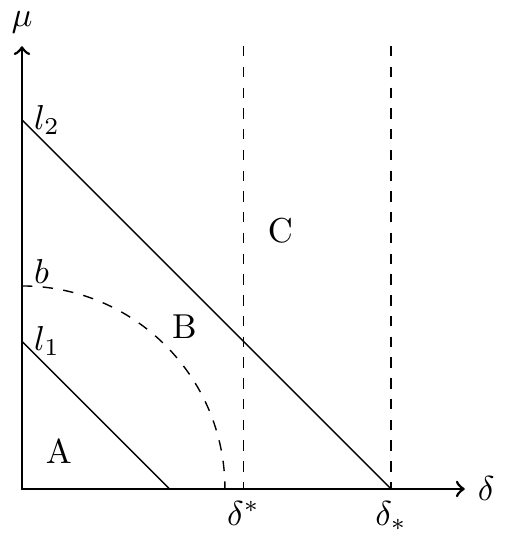}

\caption{Phase diagram of the the mutating contact process on $\TD$. We show
that in the area A (below line $l_{1}$) there is weak survival of
a single mutant , and in Area C (above line $l_{2}$) there is extinction
of a single mutant. In between the two lines (area B) a single phase
transition must occur, depicted as the dashed curve $b$. The entire
process survives strongly if $\delta<\delta_{*}$, dies out if $\delta<\delta^{*}$
and survives weakly in between. Note that $l_{2}$ indeed itersects
the $\delta$ axis at $\delta_{*}$ as shown, while the relation of
$\delta^{*}$ to the weak survival phase is unknown.}
\end{figure}
The final results which is needed for drawing a qualitative phase
diagram of the mutating contact on $S=\TD$ is that if $\delta+\mu$
is sufficiently large then even weak survival of 1s is impossible.
This is the next proposition, the proof of which follows from a coupling
with a dominating one type contact process which also does not survive
weakly. 
\begin{prop}
\label{prop:dominated death}For the mutating contact process with
on $\TD$ with any initial finite configuration of 1s, a single mutant
dies out, 
\[
P\left(^{1}{{\xi}_{t}}\ne\varnothing,\,\forall t>0\right)=0,
\]
if $\delta+\mu>{{\delta}_{*}}$ where ${{\delta}_{*}}$ is the upper
critical value.
\end{prop}

We wish to comment that since the process is monotone (an instance
of the model with smaller values of $\delta$ or $\mu$ can be attained
by thining the corresponding Poisson processes, see the graphical
construction in the next section for the details), there can be only
a single phase transition. Finally, note that the support of ${{\xi}_{t}}$
on $\TD$ follows the evolution of the single type contact process
with death rate $\delta$ and birth rate 1, as mutation events are
muted through this projection. Therefore, while the survival of a
single strain depends on both $\delta$ and $\mu$, the survival of
the entire species depends only on the value of $\delta$ in relation
to the two critical values ${{\delta}^{*}}<{{\delta}_{*}}$, which
will be defined in the next section.

\section{Previous Results and Model Construction}

Before we continue, it will be useful to recall basic results concerning
the one-type contact process. Please note that in the mutating contact
process, the free parameter is the death rate $\delta$, while the
birth rate is set at 1. It is usually the opposite case in previous
results: the death rate is set to 1 while the birth rate is a free
parameter $\lambda$. Through a rescaling of time by a factor of $1/\delta$,
one may transform the first into the second with $\lambda=1/\delta$
while maintaining equality of the law of the one-type contact process.
Thus, most results transfer directly, mutatis mutandis. For convenience
we will restate the results directly translated through this transformation. 

We will denote the set of occupied vertices of a one type contact
process on $S$ with death rate $\delta$ and birth rate 1 by ${{\zeta}_{t}}$,
and if it is initialized at a single site $x$, then we will denote
it $\zeta_{t}^{x}$. With this notation we will define two critical
values, 
\begin{align*}
 & {{\delta}_{*}}=\sup\left\{ \delta:P\left(\zeta_{t}^{x}\ne\varnothing\,\,\forall t>0\right)>0\right\} ,\\
 & {{\delta}^{*}}=\sup\left\{ \delta:P\left(x\in\zeta_{t}^{x}\text{ i}\text{.o as }t\to\infty\right)>0\right\} .
\end{align*}
 These critical values partition the survival of the contact process
into the following phases:
\begin{enumerate}
\item If $\delta<{{\delta}^{*}}$ then $P\left(x\in\zeta_{t}^{x}\text{ i}\text{.o as }t\to\infty\right)>0$and
we say that the process survives strongly, 
\item If ${{\delta}^{*}}<\delta<{{\delta}_{*}}$ then $P\left(x\in\zeta_{t}^{x}\text{ i}\text{.o as }t\to\infty\right)=0$
but $P\left(\zeta_{t}^{x}\ne\varnothing\,\,\forall t>0\right)>0$
and we say that the process survives weakly. 
\item If ${{\delta}_{*}}<\delta$ then $P\left(\zeta_{t}^{x}\ne\varnothing\,\,\forall t>0\right)=0$
and we say that the process dies out.
\end{enumerate}
It has been proven that on $S=\ZD$, ${{\delta}^{*}}={{\delta}_{*}}<\infty$,
so that weak survival is an impossibility and that for $S=\TD$, $0<{{\delta}^{*}}<{{\delta}_{*}}<\infty$
and the phases are right continuous at the critical values, See chapter
I.4 in Ligget \cite{liggett2013stochastic}. 

The \emph{complete convergence theorem} states that no matter the
initial configuration ,$\ensuremath{{{\zeta}_{t}}}$ converges weakly
to a mixture of the point mass on the empty configuration, ${{\delta}_{\varnothing}}$,
and another measure called the upper invariant measure, denoted by
$\bar{\nu}$. In details, denoting by $\alpha$ the probability of
survival (which obviously depends on the initial configuration, that
is, $\alpha=P\left(\zeta_{t}^{x}\ne\varnothing\,\,\forall t>0\right)$,
then the complete convergence theorem states that, ${{\zeta}_{t}}\Rightarrow\left(1-\alpha\right){{\delta}_{\varnothing}}+\alpha\bar{\nu}$.
Confer chapter I.4 of Ligget \cite{liggett2013stochastic} for the
details. 

\subsection*{Model Construction.}

The graphical construction of the contact process is a method of realizing
the contact process through the use of independent Poisson processes,
introduced by Harris \cite{Harris1974}. The construction applied
to the mutating contact processes is as follows. For each vertex$x\in S$
define two Poisson processes independent of all other processes : 
\begin{enumerate}
\item The first of these will have constant rate $\delta$ and will be appropriately
named the \emph{death} process at $x$. Denote its arrival times by
$\left\{ \delta_{n}^{x}:n\ge1\right\} $.
\item The second will have constant rate $\mu$ and will be named the \emph{mutation}
process at $x$. Denote its arrival times by $\left\{ \mu_{n}^{x}:n\ge1\right\} $.
\end{enumerate}
Further, for each ordered pair of neighboring vertices, $\left(x,y\right)\in S\times S$,
define independently another Poisson process, with rate 1, and name
it the \emph{infection} process from $x$ to$y$. Denote its arrival
times by $\left\{ {{\left(x\to y\right)}_{n}}:n\ge1\right\} $.

For some vertices $\left(x,y\right)\in S$ and times $s\le t$, we
say that there is an open 1-path from $\left(x,s\right)$ (read -
$x$ at time $s$) to $\left(y,t\right)$ if one can connect the two
space-time values by use of the timelines attached to any vertex and
outgoing arrows without going back in time and without crossing any
arrival of the death or mutation processes. 

More formally, there is open 1-path from $\left(x,s\right)$ to $\left(y,t\right)$
if there exists a sequence of space-time values $\left\{ \left({{x}_{i}},{{t}_{i}}\right)\right\} _{_{i=0}}^{n}$
such that: 
\begin{enumerate}
\item The sequence starts at $\left({{x}_{0}},{{t}_{0}}\right)=\left(x,s\right)$
and ends with $\left({{x}_{n}},{{t}_{n}}\right)=\left(y,t\right)$, 
\item The time values are increasing, ${{t}_{0}}<{{t}_{1}}<...<{{t}_{n}}$,
\item Consecutive vertices are neighbors, ${{x}_{i}}\sim{{x}_{i-1}},\,\,\,i\ge1$. 
\item For $0<i<n$, the time ${{t}_{i}}$ is an arrival of the infection
process from ${{x}_{i-1}}$ to ${{x}_{i}}$, ${{t}_{i}}={{\left({{x}_{i-1}}\to{{x}_{i}}\right)}_{k}}$
for some $k$. 
\item For $0\le i<n$, there is no arrival of the death process at ${{x}_{i}}$
during $\left({{t}_{i}},{{t}_{i+1}}\right)$. 
\item For $0\le i<n$, there is no arrival of the mutation process at ${{x}_{i}}$
during $\left({{t}_{i}},{{t}_{i+1}}\right)$. 
\end{enumerate}
Similarly we will say that there is an open 2-path from $\left(x,s\right)$
to $\left(y,t\right)$ if an appropriate sequence obeying 1-5 of the
above exists. By this definition, every 1-path is immediately also
a 2-path. 

Given an initial configuration , one may evolve it according to the
above graphical construction by propagating initial 1s through open
1-paths and initial 2s through open 2-paths. One must take care to
change 1s into 2s whenever they encounter a mutation arrival at their
site, and continue to propagate the resulting 2s through open 2-paths. 

The resulting configuration at time $t$ obeys the law of the mutating
contact process and therefore may be rightly denoted ${{\xi}_{t}}$.
We will omit the proof of this assertion, and defer to Harris \cite{Harris1974}.

\subsection*{Known Results}

Two results from \cite{Andjel2010} are used in the proof of Theorem
\ref{thm: no strong survival on Z}. We bring them here for the convenience
of the reader, while rephrasing them with our notation. The first
is Corollary 2.4.
\begin{cor*}
\textbf{\label{(=000023Andjel=000023) speed}} Suppose $\delta<{{\delta}_{*}}$
and let $A$ be an infinite subset of the positive integers, $\mathbb{Z}_{+}$,
then there exists a constant $v>0$, such that for any ${v}'<v$ there
exists some $x\in A$ with an open 2-path which is to the right of
the line $\left\{ \left({v}'t,t\right):t\ge0\right\} $.
\end{cor*}
Note that $v$ is the speed of propagation of the one-type contact
process, and is dependent on (and negatively correlated to) the death
rate. The second result required is Theorem 3.9.
\begin{thm*}
\textbf{\label{thm:(=000023Andjel=000023) 1s surrouned by 2s}} Consider
the two-type contact process ${{\eta}_{t}}$ on $\mathbb{Z}$ with
initial configuration ${{\eta}_{0}}$ such that there is a finite
number of 1s bounded by an infinite number of 2s on both sides. Then
\[
P\left(^{1}{{\xi}_{t}}\ne\varnothing,\,\forall t>0\right)=0.
\]
\end{thm*}
The last result we wish to cite is Proposition 1 from \cite{Cox2009},
which will be used in the proof of Theorem \ref{thm:1 dies on TD}. 
\begin{prop*}
Assume$S$ is $\ZD$ or $\TD$ and $\bar{\nu}$ is the upper invariant
measure of the one-type contact process with death rate $\delta<{{\delta}_{*}}$.
Define 
\[
{{\delta}_{L}}=\underset{\left|A\right|>L}{\sup}\bar{\nu}\left(\zeta:\zeta\cap A=\varnothing\right).
\]
Then ${{\delta}_{L}}\to0$ as $L\to\infty$. 
\end{prop*}

\section{Proof of Results}

In this section we will prove Theroems \ref{thm: no strong survival on Z}
and \ref{thm:1 dies on TD}, and propositions \ref{prop:GW survival}
and \ref{thm: no strong survival on Z}. 
\begin{proof}[Proof of Theorem \ref{thm: no strong survival on Z}]
We will consider the case where $\delta<{{\delta}_{*}}$ such that
the one-type contact process has a positive probability of survival.
As well, it suffices to consider all probabilities on the event of
species survival. The certainty of the extinction of a single mutant
is trivial in the other cases. 

Denote the leftmost site occupied by a 1 at time $t$ by ${{L}_{t}}=\min\left\{ x:{{\xi}_{t}}\left(x\right)=1\right\} $
and the rightmost site by ${{R}_{t}}=\max\left\{ x:{{\xi}_{t}}\left(x\right)=1\right\} $.
${{L}_{t}}$ and ${{R}_{t}}$ both undergo mutation at a constant
rate (without relation to the specific spatial location) so that if
${{s}_{0}}$ is an arrival time of such a mutation, then a site to
the right of ${{R}_{{{s}_{0}}}}$ , $\underset{s\to{{s}_{0}}-}{\lim}{{R}_{s}}$,
will be occupied by a 2 at time ${{s}_{0}}$. Since there are w.p.
1 infinitely many such arrivals, the set $^{2}{{\xi}_{t}}\cap\left({{R}_{t}},\infty\right)$
is non-empty infinitely many times almost surely.

Corollary 2.4 of \cite{Andjel2010} implies that there is a positive
probability, ${{p}_{0}}$, such that an infinite open path starting
from $\left(0,0\right)$ lying on the right of the line $\left\{ \left(kt,t\right):t\ge0\right\} $
for some $k>0$ exists. From the time and space homogeneity of the
graphical construction it follows that this statement applies to any
starting point $\left(x,s\right)$ with regards to the line $\left\{ \left(x+kt,s+t\right):t\ge0\right\} $.
Thus, combined with the previous results, almost surely there exists
such a path with ${{\xi}_{t}}\left(x\right)=2$. 

By the same argument regarding ${{L}_{t}}$, almost surely there exists
an infinite open path lying on the left of $\left\{ \left(x-kt,s+t\right):t\ge0\right\} $
such that at its beginning ${{\xi}_{s}}\left(x\right)=2$. It follows
that w.p. 1 there exists a time ${{t}_{0}}$ such that for all $t\ge{{t}_{0}}$,
the set of all 1s,$^{1}{{\xi}_{t}}$, is contained in an interval
$\left(z\left(t\right),y\left(t\right)\right)$ where ${{\xi}_{t}}\left(z\left(t\right)\right)={{\xi}_{t}}\left(y\left(t\right)\right)=2$.

Now, we argue the evolution of ${{\xi}_{{{t}_{0}}}}$ inside the interval$\left(z\left(t\right),y\left(t\right)\right)$
is independent the configuration outside of the interval: for the
purpose of determining the value of ${{\xi}_{t}}$ on $\left(z\left(t\right),y\left(t\right)\right)$,
any open path from outside the interval into it, may as well be truncated
to the time-space point where it crosses $z\left(t\right)\text{ or }y\left(t\right)$,
since these endpoints are already occupied by a 2 which will propogate
along the rest of the path. Thus we may assume 
\[
{{\xi}_{{{t}_{0}}}}\left(\left(-\infty,z\left({{t}_{0}}\right)\right]\cup\left[y\left({{t}_{0}}\right),\infty\right)\right)=2
\]
 without affecting the evolution of 1s in the model. 

At this point we will couple the mutating contact process with an
instance of the standard two-type contact process ${{\eta}_{t}}$
with initial configuration ${{\eta}_{{{t}_{0}}}}={{\xi}_{{{t}_{0}}}}$
such that the following monotonicity property holds, 
\[
^{1}{{\xi}_{t}}\subseteq\left\{ x\in\mathbb{Z}:{{\eta}_{t}}\left(x\right)=1\right\} \text{ for all }t\ge{{t}_{0}}.
\]
This coupling is attained by applying the same graphical construction
to the processes from the time ${{t}_{0}}$, while ignoring the mutation
arrivals when evolving the two-type process. It is evident that for
such a coupling, starting from the same initial configuration, the
monotonicity holds: if ${{\xi}_{t}}\left(x\right)=1$ then an open
path with no mutation arrivals from an initial 1 to $\left(x,t\right)$
must exist. The same open path implies ${{\eta}_{t}}\left(x\right)=1$.
An application of Theorem 3.9 of \cite{Andjel2010} on this instance
of the two-type process with its initial conditions yields 
\[
P\left(\left\{ x:{{\eta}_{t}}\left(x\right)=1\right\} \ne\varnothing,\,\forall t>0\right)=0
\]
 which by the monotonicity property also gives 
\[
P\left(^{1}{{\xi}_{t}}\ne\varnothing,\,\forall t>0\right)=0,
\]
as required. 
\end{proof}
~
\begin{proof}[Proof of Theorem \ref{thm:1 dies on TD}.]
We note that $^{1}{{\xi}_{t}}$ can be coupled to a one type contact
process with birth rate 1 and death rate $\delta+\mu$ by using the
same graphical construction as ours and treating mutation arrivals
as death arrivals. We will denote this coupled process by by $^{1}{{\zeta}_{t}}$.
In this coupling it is true $^{1}{{\xi}_{t}}\subseteq{}^{1}{{\zeta}_{t}}$
for the same initial configuration of 1s, as both propagate through
the same open 1-paths with the added restriction for the first that
some vertices along those open paths may already be occupied by 2s.
Thus, if $\delta+\mu>{{\delta}^{*}}$ then $P\left({{\xi}_{t}}\left(x\right)=1\right)\to0$
follows immediately from the lack of strong survival of the coupled
one-type process.

If it is true that $\delta+\mu<{{\delta}^{*}}$, the proof becomes
more complicated and is inspired by the proof of Theorem 2 in \cite{Cox2009}.
Before we begin the proof in this case, we may as well discuss the
matter conditioned on the event of weak survival of type 1, as well
as assuming that weak survival of type 1 may occur with positive probability.
The result is obvious in the other cases. In other words, throughout
this proof we consider all probabilities to be conditional on the
event $\left\{ ^{1}{{\xi}_{t}}\ne\varnothing,\,\forall t>0\right\} $.
On this event, it also follows that $\left|^{2}{{\xi}_{s}}\right|\ne\varnothing$
infinitely often, as type 1 mutates into type 2 at a constant rate.
Since individuals die at a constant rate, the descendants of any single
individual must either die out or tend to infinity as $t\text{ }\to\text{ }\infty$.
This follows from the fact that with some constant positive probability,
$M$ individuals (or less) will die before propagating to other vertices
in a single time unit. If at an infinite and unbounded set of times
$t$ the number of descendants is bounded above by some fixed $M>0$,
then as $t\to\infty$, the (at most) $M$ individuals will die out
with probability tending to 1 as a result of geometric sampling of
the aforementioned event. Thus, under the assumption that $\left\{ ^{1}{{\xi}_{t}}\ne\varnothing,\,\forall t>0\right\} $
occurs, it follows that for fixed $M>0$
\begin{equation}
P\left(\left|^{1}{{\xi}_{t}}\right|\ge M\right)\to1\text{ as }\text{ }t\text{ }\to\text{ }\infty,
\end{equation}
which also implies that for fixed $L>0$,
\begin{equation}
P\left(\left|^{2}{{\xi}_{t}}\right|\ge L\right)\to1\text{ as }\text{ }t\text{ }\to\text{ }\infty\label{eq:2>L}
\end{equation}
since 1s mutate to 2s at a constant rate. 

Let $x\in S$ be some vertex. We wish to show that $P\left({{\xi}_{t}}\left(x\right)=1\right)\to0$
as $\text{ }t\text{ }\to\text{ }\infty$. Using the previous result,
we may restrict the events in question to $\left\{ {{\xi}_{t+u}}\left(x\right)=1,\left|^{2}{{\xi}_{u}}\right|\ge L\right\} ,$which
are subsets of $\left\{ {{\xi}_{t+u}}\left(x\right)=1\right\} $ and
show 
\[
\underset{u\to\infty}{\lim}\underset{t\to\infty}{\limsup}P\left({{\xi}_{t+u}}\left(x\right)=1,\left|^{2}{{\xi}_{u}}\right|\ge L\right)=0.
\]
To this end we need to express the event 
\[
\left\{ {{\xi}_{t+u}}\left(x\right)=1,\left|^{2}{{\xi}_{u}}\right|\ge L\right\} 
\]
 with the notation of the ancestor process, defined in a similar fashion
to \cite{Cox2009}. Let an ancestor configuration $\hat{\xi}$ be
a (possibly empty) sequence of pairs $\left(\left({{a}_{1}},{{b}_{1}}\right),...,\left({{a}_{n}},{{b}_{n}}\right)\right)$
for some $n\ge1$, where each ${{a}_{j}}\in S$ and ${{b}_{j}}\in\left\{ 1,2\right\} $.
The vertices ${{a}_{j}}$ will denote the possible ancestors in descending
order of precedence, and ${{b}_{j}}$ will indicate whether a certain
possible ancestor is viable only if it is occupied by a 2 (2-viable)
or if it is viable with no regard to its type (in which case we will
say it is 1-viable).

In more details, we define the ancestor process $\hat{\xi}_{s}^{\left(x,t\right)}$
to be the list of possible ancestors of $\left(x,t\right)$ at time
$t-s$ for some $s\in\left[0,t\right]$. $\hat{\xi}_{s}^{\left(x,t\right)}$
is a Markov process defined recursively as follows: At time $t$ the
only possible ancestor of $\left(x,t\right)$ is $x$ itself, so that
$\hat{\xi}_{0}^{\left(x,t\right)}=\left(\left(x,1\right)\right)$.
Now, under the assumption that $\hat{\xi}_{u}^{\left(x,t\right)}$
has been properly defined for all $u\le s$:
\begin{enumerate}
\item If $\hat{\xi}_{s}^{\left(x,t\right)}$ is empty, define $\hat{\xi}_{v}^{\left(x,t\right)}$
to be empty as well for all $v\in\left(s,t\right]$. 
\item Otherwise, if $\hat{\xi}_{s}^{\left(x,t\right)}=\left(\left({{a}_{1}},{{b}_{1}}\right),...,\left({{a}_{n}},{{b}_{n}}\right)\right)$,
then observing the graphical construction, look for the most recent
past event affecting any of the ${{a}_{j}}$ (an arrival of any of
the Poisson processes involved). If no such event occurs till time
0 then define $\hat{\xi}_{v}^{\left(x,t\right)}=\hat{\xi}_{s}^{\left(x,t\right)}$
for all $v\in\left(s,t\right]$. 
\item If Such an event occurs, denote its time of occurrence by $r$ and
define $\hat{\xi}_{v}^{\left(x,t\right)}=\hat{\xi}_{s}^{\left(x,t\right)}$
for all $v\in\left(s,r\right)$. $\hat{\xi}_{r}^{\left(x,t\right)}$
itself is defined in accordance with the specific event that occurred: 
\begin{enumerate}
\item If the event is an arrow pointing from some $a\in S$ to ${{a}_{j}}$
then insert $\left(a,b\right)$ into the sequence $\hat{\xi}_{s}^{\left(x,t\right)}$
immediately after each appearance of ${{a}_{j}}$ and set the value
of the type viability $b$ to be identical to that of the preceding
appearance of ${{a}_{j}}$ (i.e. if ${{b}_{j}}=x$ set $b=x$ only
for the following insertion). 
\item If the event is a mutation arrival at ${{a}_{j}}$, set ${{b}_{j}}=2$
for all appearances of ${{a}_{j}}$ in the sequence (if it isn\textquoteright t
set already).
\item If the event is a death arrival, delete all appearances of ${{a}_{j}}$
from the sequence. If all the ${{a}_{j}}$ have been deleted, set
$\hat{\xi}_{r}^{\left(x,t\right)}$ to be empty and proceed as in
1.
\end{enumerate}
\end{enumerate}
This algorithm defines $\hat{\xi}_{s}^{\left(x,t\right)}$ in such
a way that ${{\xi}_{t}}\left(x\right)$ takes the value of the first
vertex in $\hat{\xi}_{s}^{\left(x,t\right)}$ which is occupied at
time $t-s$ (conditional on its 1-viability, if the vertex is occupied
by a 1). The ancestor process $\hat{\xi}_{s}^{\left(x,t\right)}$
moves backwards in time and so can only be defined up to time $\text{s=}t$,
that is for bounded intervals. To overcome this difficulty, using
the time reversibility of the underlying Poisson processes, we can
switch to forward time as in \cite{Cox2009,Neuhauser1992}, and define
a variation of the process $\overline{\xi}_{s}^{\left(x,t\right)}$
for all $s\ge0$ such that the laws of $\overline{\xi}_{s}^{\left(x,t\right)}$
and $\hat{\xi}_{s}^{\left(x,t\right)}$ agree for all $s\le t$. Henceforward,
if the ancestral process is strated at $\left(x,0\right)$, we will
denote it $\overline{\xi}_{s}^{x}$.

We may now return to discuss the events $\left\{ {{\xi}_{t+u}}\left(x\right)=1,\left|^{2}{{\xi}_{u}}\right|\ge L\right\} $.
First assume that ${{\xi}_{u}}=\eta$ where $\eta$ is some non-random
configuration of 1s and 2s on $S$ with $\left|^{2}\eta\right|\ge L$.
We will later use the law of total probability to generalize the following
argument. Note that ${{\xi}_{t+u}}\left(x\right)=1$ if and only if
the following event, denoted $OC\left(\hat{\xi}_{t}^{\left(x,t+u\right)},\eta\right)$,
occurs: the first occupied vertex in the ancestor list $\hat{\xi}_{t}^{\left(x,t+u\right)}$
is 1-viable and is actually occupied by a 1 in the configuration $\eta$.
Since $\hat{\xi}_{t}^{\left(x,t+u\right)}$ has the same law as $\overline{\xi}_{t}^{x}$
it follows that 
\begin{equation}
P\left(OC\left(\hat{\xi}_{t}^{\left(x,t+u\right)},\eta\right)\right)=P\left(OC\left(\overline{\xi}_{t}^{x},\eta\right)\right).
\end{equation}
.

Define the list of first priority 2-viable ancestors, $A_{t}^{x}$
to be empty if the first ancestor in the list $\overline{\xi}_{t}^{x}$
is 1-viable, otherwise define it to be the set of all of the 2-viable
ancestors that appear on the list before the first 1-viable ancestor.

Now, from these definitions we can deduce that that $OC\left(\overline{\xi}_{t}^{x},\eta\right)$
does not co-occur with $A_{t}^{x}\cap{}^{2}\eta\ne\varnothing$ and
that $OC\left(\overline{\xi}_{t}^{x},\eta\right)\subseteq\left\{ \overline{\xi}_{t}^{x}\ne\varnothing\right\} $.
Thus,
\begin{equation}
P\left(OC\left(\overline{\xi}_{t}^{x},\eta\right)\right)\le P\left(\overline{\xi}_{t}^{x}\ne\varnothing,A_{t}^{x}\cap{}^{2}\eta=\varnothing\right).
\end{equation}
 .

Summing over all configurations $\eta$ and using the independence
of disjoint time-space regions we get 
\begin{align}
 & P\left({{\xi}_{t+u}}\left(x\right)=1,\left|^{2}{{\xi}_{u}}\right|\ge L\right)\label{eq:total prob}\\
= & \int\limits _{\left|^{2}\eta\right|\ge L}{P\left({{\xi}_{u}}\in d\eta\right)P\left(OC\left(\hat{\xi}_{t}^{\left(x,t+u\right)},\eta\right)\right)}\nonumber \\
= & \int\limits _{\left|^{2}\eta\right|\ge L}{P\left({{\xi}_{u}}\in d\eta\right)P\left(OC\left(\overline{\xi}_{t}^{x},\eta\right)\right)}\nonumber \\
\le & \int\limits _{\left|^{2}\eta\right|\ge L}{P\left({{\xi}_{u}}\in d\eta\right)P\left(\overline{\xi}_{t}^{x}\ne\varnothing,A_{t}^{x}\cap{}^{2}\eta=\varnothing\right)}.\nonumber 
\end{align}

In what follows it will be shown that for large $t$ and $u$, very
rarely do the events $\left\{ \overline{\xi}_{t}^{x}\ne\varnothing\right\} $
and $\left\{ A_{t}^{x}\cap{}^{2}\eta=\varnothing\right\} $ co-occur.
For this purpose, we define a combination of events that render a
location only 2-viable. Call a space-time point $\left(y,s\right)$
good if a mutation arrival occurs at $y$ during $\left(s,s+1\right)$
and no other event affecting $y$ occur during this time frame. Then
${{p}_{0}}=P\left(\left(y,s\right)\text{ is good}\right)$ is a positive
constant that does not depend on $\left(y,s\right)$. Considering
a site $x$ with non-empty ancestry at time $s$: $\overline{\xi}_{s}^{x}=\left(\left({{a}_{1}}\left(s\right),{{b}_{1}}\left(s\right)\right),...,\left({{a}_{n}}\left(s\right),{{b}_{n}}\left(s\right)\right)\right),$
if it occurs that $\left({{a}_{1}}\left(s\right),s\right)$ is good
then ${{a}_{1}}\left(s\right)$ remains the primary ancestor at least
until time $s+1$ because no deaths occur and it becomes 2-viable
because of the mutation: $\left({{a}_{1}}\left(s+1\right),{{b}_{1}}(s+1)\right)=\left({{a}_{1}}\left(s\right),2\right)$.
Therefore, for all $v\ge s+1$, The forward time ancestry list of
$x$, $\overline{\xi}_{v}^{x}$, begins with the list of ancestors
of ${{a}_{1}}\left(s+1\right)$: $\overline{\xi}_{v}^{\left({{a}_{1}}\left(s+1\right),s+1\right)}$,
and every ancestor from this list is 2-vialbe for $x$.

We will now turn to the geometric sampling argument mentioned at the
beginning of the proof. Fix some $T>0$ and define ${{s}_{k}}=k\left(T+1\right)$
, ${{t}_{k}}={{s}_{k}}+1$,$k\ge0$. Let 
\begin{equation}
R=\inf\left\{ k|\overline{\xi}_{{{s}_{k}}}^{x}\ne\varnothing,\left({{a}_{1}}\left({{s}_{k}}\right),{{s}_{k}}\right)\text{ is good and }\overline{\xi}_{{{s}_{k+1}}}^{\left({{a}_{1}}\left({{t}_{k}}\right),{{t}_{k}}\right)}\ne\varnothing\right\} .
\end{equation}
$R$ is the first time that $\left({{a}_{1}}\left({{s}_{k}}\right),{{s}_{k}}\right)$
is good and that ${{a}_{1}}\left({{s}_{k}}\right)$, which is also
${{a}_{1}}\left({{t}_{k}}\right)$, has an ancestry that lasts at
least $T$ time units further than ${{t}_{k}}$. To see that $R$
is dominated by a Geometric r.v., consider from the self-duality of
the contact process that $P\left(\overline{\xi}_{{{s}_{k+1}}}^{\left({{a}_{1}}\left({{t}_{k}}\right),{{t}_{k}}\right)}\ne\varnothing\right)\ge\alpha$,
where $\alpha$ is the probability of weak survival in the coupled
contact process, $^{2}{{\zeta}_{t}}$ with initial configuration ${{a}_{1}}\left({{t}_{k}}\right)=2$.
Also consider that the events $\left\{ \left({{a}_{1}}\left({{s}_{k}}\right),{{s}_{k}}\right)\text{ is good}\right\} $
and $\left\{ \overline{\xi}_{{{s}_{k+1}}}^{\left({{a}_{1}}\left({{t}_{k}}\right),{{t}_{k}}\right)}\ne\varnothing\right\} $
are determined by disjoint regions of space-time and are thus independent.
Deduce that 
\begin{align}
 & P\left(\left({{a}_{1}}\left({{s}_{k}}\right),{{s}_{k}}\right)\text{ is good and }\overline{\xi}_{{{s}_{k+1}}}^{\left({{a}_{1}}\left({{t}_{k}}\right),{{t}_{k}}\right)}\ne\varnothing\right)\\
= & P\left(\left({{a}_{1}}\left({{s}_{k}}\right),{{s}_{k}}\right)\text{ is good }\right)P\left(\overline{\xi}_{{{s}_{k+1}}}^{\left({{a}_{1}}\left({{t}_{k}}\right),{{t}_{k}}\right)}\ne\varnothing\right)\nonumber \\
\ge & {{p}_{0}}\alpha.\nonumber 
\end{align}
 Now, This inequality holds for any $k$, so that by iteration and
the Markov property, 
\begin{equation}
P\left(\text{ }\overline{\xi}_{{{s}_{k+1}}}^{x}\ne\varnothing\text{ and }R>k\right)\le{{\left(1-\alpha{{p}_{0}}\right)}^{k}},\label{eq:R1}
\end{equation}
 consequently, for any ${{k}_{0}}>0$ and $t>{{s}_{{{k}_{0}}+1}}$
we may partition the event 
\[
\left\{ \overline{\xi}_{t}^{x}\ne\varnothing,A_{t}^{x}\cap{}^{2}\eta=\varnothing\right\} 
\]
 by the value of $R$: 
\begin{equation}
P\left(\overline{\xi}_{t}^{x}\ne\varnothing,A_{t}^{x}\cap{}^{2}\eta=\varnothing\right)\le{{\left(1-\alpha{{p}_{0}}\right)}^{{{k}_{0}}}}+\sum\limits _{k=0}^{{{k}_{0}}}{P\left(R=k,\overline{\xi}_{t}^{x}\ne\varnothing,A_{t}^{x}\cap{}^{2}\eta=\varnothing\right)}.\label{eq:R2}
\end{equation}
 We will further partition the events under the sum by the following,
defined for all $a\in S$ , 
\begin{equation}
{{G}_{k}}\left(a\right)=\left\{ R>k-1,\,\,\overline{\xi}_{{{s}_{k}}}^{x}\ne\varnothing,\,\,{{a}_{1}}\left({{s}_{k}}\right)=a\text{ and }\left(a,{{s}_{k}}\right)\text{ is good}\right\} .\label{eq:R3}
\end{equation}
 ${{G}_{k}}\left(a\right)$ is the event that the first time that
$\left({{a}_{1}}\left({{s}_{m}}\right),{{s}_{m}}\right)$ is good
is attained at $m=k$ and it is attained at ${{a}_{1}}\left({{s}_{m}}\right)=a$.
Note that for all $k\le{{k}_{0}}$ and on the event ${{G}_{k}}\left(a\right)$,
the list of first priority 2-viable ancestors of $x$ at time $t$
,$A_{t}^{x}$, contains the vertices listed in $\overline{\xi}_{t}^{\left(a,{{t}_{k}}\right)}$
because $\left(a,{{s}_{k}}\right)$ is good, so that its ancestors
are first in the ancestry list of $x$, and they must be 2-viable
for $x$ because of the mutation arrival that must occur. Denoting
the actual vertices listed (instead of the pairs $\left(a,b\right)$)
by $\text{supp}\left(\overline{\xi}_{t}^{\left(a,{{t}_{k}}\right)}\right)$
, we incorporate this relation into the partition, and get the following
inequality, 
\begin{align}
 & P\left(R=k,\overline{\xi}_{t}^{x}\ne\varnothing,A_{t}^{x}\cap{}^{2}\eta=\varnothing\right)\label{eq:R4}\\
\le & \sum\limits _{a\in S}{P\left({{G}_{k}}\left(a\right)\cap\left\{ \overline{\xi}_{{{s}_{k+1}}}^{\left(a,{{t}_{k}}\right)}\ne\varnothing,\,\,\text{supp}\left(\overline{\xi}_{t}^{\left(a,{{t}_{k}}\right)}\right)\cap{}^{2}\eta=\varnothing\right\} \right)}\nonumber 
\end{align}
and apply the independence of disjoint space-time regions on each
summand: 
\[
\begin{array}{cc}
\\
=
\end{array}
\]
\begin{align}
 & P\left({{G}_{k}}\left(a\right)\cap\left\{ \overline{\xi}_{{{s}_{k+1}}}^{\left(a,{{t}_{k}}\right)}\ne\varnothing,\,\,\text{supp}\left(\overline{\xi}_{t}^{\left(a,{{t}_{k}}\right)}\right)\cap{}^{2}\eta=\varnothing\right\} \right)\label{eq:R5}\\
= & P\left({{G}_{k}}\left(a\right)\right)P\left(\overline{\xi}_{{{s}_{k+1}}}^{\left(a,{{t}_{k}}\right)}\ne\varnothing,\,\,\text{supp}\left(\overline{\xi}_{t}^{\left(a,{{t}_{k}}\right)}\right)\cap{}^{2}\eta=\varnothing\right).\nonumber 
\end{align}
 Now, since the set $\text{supp}\left(\overline{\xi}_{t}^{\left(x,s\right)}\right)$
is not affected by mutation arrivals in the graphical construction
(they only change the second term of the pairs in $\overline{\xi}_{t}^{\left(a,{{t}_{k}}\right)}$),
it follows the evolution of a time-reversal of the regular one-type
contact process, and therefore has the same law as $^{2}{{\zeta}_{t-s}}$
initialized at $^{2}{{\zeta}_{0}}=\left\{ x\right\} $ and null elsewhere.
From this we deduce the following, while noting that ${{s}_{k+1}}-{{t}_{k}}=T$
and that $t>{{s}_{{{k}_{0}}+1}}>{{t}_{k}}$ .
\begin{equation}
P\left(\overline{\xi}_{{{s}_{k+1}}}^{\left(a,{{t}_{k}}\right)}\ne\varnothing,\text{ supp}\left(\overline{\xi}_{t}^{\left(a,{{t}_{k}}\right)}\right)\cap{}^{2}\eta=\varnothing\right)=P\left(\left|^{2}{{\zeta}_{T}}\right|\ge1,{}^{2}{{\zeta}_{t-{{t}_{k}}}}\left(^{2}\eta\right)\equiv0\right)\label{eq:R6}
\end{equation}
 This last event may occur in two distinct ways: either the entire
contact process dies out between time $T$ and time $t-{{t}_{k}}$,
or the process lives till time $t-{{t}_{k}}$ but does not appear
on the vertices in $^{2}\eta$. Denote the probability of the first
by $\rho\left(T\right)$ and observe that by the complete convergence
theorem, $\rho\left(T\right)\to0$ as $T\to\infty$. We get the following
partition, 
\begin{equation}
P\left(\left|^{2}{{\zeta}_{T}}\right|\ge1,{}^{2}{{\zeta}_{t-{{t}_{k}}}}\left(^{2}\eta\right)\equiv0\right)=\rho\left(T\right)+P\left(\left|^{2}{{\zeta}_{t-{{t}_{k}}}}\right|\ge1,{}^{2}{{\zeta}_{t-{{t}_{k}}}}\left(^{2}\eta\right)\equiv0\right).\label{eq:R7}
\end{equation}

Combining \ref{eq:R2} to \ref{eq:R7}, we have,

\begin{align}
 & P\left(\overline{\xi}_{t}^{x}\ne\varnothing,A_{t}^{x}\cap{}^{2}\eta=\varnothing\right)\label{eq:BIGINEQ}\\
\le & {{\left(1-\alpha{{p}_{0}}\right)}^{{{k}_{0}}}}+\sum\limits _{k=0}^{{{k}_{0}}}{P\left(R=k,\overline{\xi}_{t}^{x}\ne\varnothing,A_{t}^{x}\cap{}^{2}\eta=\varnothing\right)}\nonumber \\
\le & {{\left(1-\alpha{{p}_{0}}\right)}^{{{k}_{0}}}}+\sum\limits _{k=0}^{{{k}_{0}}}{\sum\limits _{a\in S}^ {}{P\left({{G}_{k}}\left(a\right)\cap\left\{ \overline{\xi}_{{{s}_{k+1}}}^{\left(a,{{t}_{k}}\right)}\ne\varnothing,\text{ supp}\left(\overline{\xi}_{t}^{\left(a,{{t}_{k}}\right)}\right)\cap{}^{2}\eta=\varnothing\right\} \right)}}\nonumber \\
= & {{\left(1-\alpha{{p}_{0}}\right)}^{{{k}_{0}}}}+\sum\limits _{k=0}^{{{k}_{0}}}{\sum\limits _{a\in S}^ {}{P\left({{G}_{k}}\left(a\right)\right)P\left(\overline{\xi}_{{{s}_{k+1}}}^{\left(a,{{t}_{k}}\right)}\ne\varnothing,\text{ supp}\left(\overline{\xi}_{t}^{\left(a,{{t}_{k}}\right)}\right)\cap{}^{2}\eta=\varnothing\right)}}\nonumber \\
= & {{\left(1-\alpha{{p}_{0}}\right)}^{{{k}_{0}}}}+\sum\limits _{k=0}^{{{k}_{0}}}{\sum\limits _{a\in S}^ {}{P\left({{G}_{k}}\left(a\right)\right)P\left(\left|^{2}{{\zeta}_{T}}\right|\ge1,{}^{2}{{\zeta}_{t-{{t}_{k}}}}\left(^{2}\eta\right)\equiv0\right)}}\nonumber \\
\le & {{\left(1-\alpha{{p}_{0}}\right)}^{{{k}_{0}}}}+\sum\limits _{k=0}^{{{k}_{0}}}{\sum\limits _{a\in S}^ {}{P\left({{G}_{k}}\left(a\right)\right)\left(\rho\left(T\right)+P\left(\left|^{2}{{\zeta}_{T}}\right|\ge1,{}^{2}{{\zeta}_{t-{{t}_{k}}}}\left(^{2}\eta\right)\equiv0\right)\right)}}\nonumber \\
\le & {{\left(1-\alpha{{p}_{0}}\right)}^{{{k}_{0}}}}+\left(k_{0}+1\right)\rho\left(T\right)+\sum\limits _{k=0}^{{{k}_{0}}}P\left(\left|^{2}{{\zeta}_{t-{{t}_{k}}}}\right|\ge1,{}^{2}{{\zeta}_{t-{{t}_{k}}}}\left(^{2}\eta\right)\equiv0\right).\nonumber 
\end{align}

Where the last inequality follows from the fact that the events ${{G}_{k}}\left(a\right)$
are mutually exclusive, implying $\sum\limits _{a\in S}{{{G}_{k}}\left(a\right)}\le1$.

Now, by the complete convergence theorem and Proposition 1 from \cite{Cox2009},
we have 
\begin{equation}
\underset{t\rightarrow\infty}{\limsup}P\left(\left|^{2}{{\zeta}_{t-{{t}_{k}}}}\right|\ge1,{}^{2}{{\zeta}_{t-{{t}_{k}}}}\left(^{2}\eta\right)=0\right)\le\bar{v}\left(\left\{ \gamma:\gamma{{\cap}^{2}}\eta=\varnothing\right\} \right)\le\delta_{L},\label{eq:limsup,deltaL}
\end{equation}

where ${{\delta}_{L}}$ tends to zero as $L\to\infty$(recall that
$\left|^{2}\eta\right|\ge L$).

\ref{eq:BIGINEQ} together with \ref{eq:limsup,deltaL} and Fatou\textquoteright s
lemma yield 
\begin{equation}
\underset{t\to\infty}{\limsup}P\left(\overline{\xi}_{t}^{x}\ne\varnothing,A_{t}^{x}\cap{}^{2}\eta=\varnothing\right)\le{{\left(1-\alpha{{p}_{0}}\right)}^{{{k}_{0}}}}+\left(k_{0}+1\right)\left(\rho\left(T\right)+{{\delta}_{L}}\right).\label{limsup almost done}
\end{equation}
 Taking $L\to\infty$ and using Fatou\textquoteright s lemma once
more it follows from \ref{eq:2>L}, \ref{eq:total prob} and \ref{limsup almost done}
that

\begin{align}
 & \underset{t\to\infty}{\limsup}P\left({{\xi}_{t+u}}\left(x\right)=1\right)\\
= & \underset{L\to\infty}{\limsup}\,\underset{t\to\infty}{\limsup}P\left({{\xi}_{t+u}}\left(x\right)=1,\left|^{2}{{\xi}_{u}}\right|\ge L\right)\nonumber \\
\le & \underset{L\to\infty}{\limsup}\int\limits _{\left|^{2}\eta\right|\ge L}{P\left({{\xi}_{u}}\in d\eta\right)\underset{t\to\infty}{\limsup}P\left(\overline{\xi}_{t}^{x}\ne\varnothing,A_{t}^{x}\cap{}^{2}\eta=\varnothing\right)}\nonumber \\
\le & {{\left(1-\alpha{{p}_{0}}\right)}^{{{k}_{0}}}}+\left({{k}_{0}}+1\right)\rho\left(T\right).\nonumber 
\end{align}
 Finally, taking $u\to\infty$ then $T\to\infty$ and finally ${{k}_{0}}\to\infty$
we deduce
\begin{equation}
\underset{u\to\infty}{\lim}\underset{t\to\infty}{\limsup}P\left({{\xi}_{t+u}}\left(x\right)=1\right)=0
\end{equation}
 as required.
\end{proof}
~
\begin{proof}[Proof of Proposition \ref{prop:GW survival}]
Signify some vertex ${{x}_{0}}\in{{T}_{d}}$ as the root and assume
the initial configuration is ${{\xi}_{0}}\left({{x}_{0}}\right)=1$
and zero elsewhere. With some positive probability any other finite
initial configuration with
\[
\left\{ x,{{\xi}_{0}}\left(x\right)=1\right\} \ne\varnothing
\]
 may evolve into the one above, and thus the proof with that specific
initial configuration will be applicable to all others considered.

We will consider ${{x}_{0}}$ to be at level 0 of the tree, and any
vertex $y$ that is removed by $n-1$ distinct vertices from $x_{0}$
to be at level $n$ and we will call the immediate neighbors of $y$
that are at level $n+1$ its children. Its single neighbor at level
$n-1$ will be called its parent. We will denote the level of a vertex
by $l\left(y\right)$. 

For each $y\in\TD$ , denote the first time it is infected with 1
by $t\left(y\right)$, 
\[
t\left(y\right)=\inf\left\{ t|{{\xi}_{t}}\left(y\right)=1\right\} ,
\]
and the time this first arrival of 1 dies or mutates by $s\left(y\right)$,
\[
s\left(y\right)=\inf\left\{ s>t\left(y\right)|{{\xi}_{s}}\left(y\right)\ne1\right\} .
\]
. Note that by the memorylessness property, $s\left(y\right)-t\left(y\right)$
is an exponential r.v. with parameter $\delta+\mu$, independent of
$t\left(y\right)$. 

Let $W_{x_{0}}^{0}$ be the set of all children of ${{x}_{0}}$ that
${{x}_{0}}$ infects with 1 before dying or mutating. That is, 
\[
W_{x_{0}}^{0}=\left\{ y\in\TD|l\left(y\right)=1,t\left(y\right)<s\left(x_{0}\right)\right\} 
\]
 . Further, let $W_{y}^{1}$ for each $y\in W_{x_{0}}^{0}$ be the
set of all its children that it infects before dying or mutating for
the first time, 
\[
W_{y}^{1}=\left\{ z\in\TD|l\left(y\right)=2,t\left(z\right)<s\left(y\right)\right\} .
\]
We continue recursively with these definitions further down the levels
of the tree so that for each $z\in W_{y}^{n-1}$ we have the set 
\[
W_{z}^{n}=\left\{ w\in\TD|l\left(w\right)=n+1,t\left(w\right)<s\left(z\right)\right\} .
\]
Note that $\left|W_{z}^{n}\right|$ is the number of children $z$
infects before dying or mutating for the first time. Also note that
each $\left|W_{z}^{n}\right|$ is determined by different parts of
time-space and thus are all independent. Further, each child is infected
at rate 1 while $z$ dies or mutates at rate $\delta+\mu$ and the
times (beyond $t\left(z\right)$) these occur are independent. Thus
$\left|W_{z}^{n}\right|$ has a binomial distribution with parameters
$\left(d-1,1/\left(1+\delta+\mu\right)\right)$.

From this discussion we can conclude that the following is a Galton
Watson process,
\[
{{S}_{n}}=\sum\limits _{z\in W_{z}^{n-1}}^ {}{\left|W_{z}^{n}\right|}.
\]
For small enough $\delta+\mu$, the expectation of $\left|W_{z}^{n}\right|$
will be larger than 1, so that ${{S}_{n}}$ is supercritical and thus
its survival probability is positive. Since ${{S}_{n}}$ counts some
(albeit not alll) of the vertices $x\in\TD$ such that there exists
$t>0$ with ${{\xi}_{t}}\left(x\right)=1$, and since with positive
probability ${{S}_{n}}\to\infty$, it follows that 
\[
P\left(\left\{ x,{{\xi}_{t}}\left(x\right)=1\right\} \ne\varnothing,\,\forall t>0\right)>0.
\]
\end{proof}
~
\begin{proof}[Proof of Proposition \ref{prop:dominated death}.]
Recalling from the beginning of the proof of Theorem 2 that $^{1}{{\xi}_{t}}$
can be coupled monotonically to$^{1}{{\zeta}_{t}}$ which is a one-type
process with birth rate 1 and death rate $\delta+\mu$. The result
follows when considering that $\delta+\mu\ge{{\delta}_{*}}$ implies
the a.s. extinction of $^{1}{{\zeta}_{t}}$ and that of $^{1}{{\xi}_{t}}$
follows by the monotonicity.
\end{proof}

\section{Notation}

In this section we will provide a summary of the notation used throughout
the paper for reference.

\subsection*{General notation.}
\begin{description}
\item [{$S$,}] the graph on which our processes are defined.
\item [{${{\xi}_{t}}$,}] the mutating contact process.
\item [{$\delta$,}] the per capita death rate, a free parameter.
\item [{$^{i}{{\xi}_{t}}=\left\{ x\in S|{{\xi}_{t}}\left(x\right)=i\right\} $,}] the
set of all vertices occupied by type $i$. in the mutating contact
process.
\item [{${{\zeta}_{t}}$,}] the one type contact process with death rate
$\delta$ and birth rate 1.
\item [{$^{1}{{\zeta}_{t}}$,}] a one-type process coupled to $^{1}{{\xi}_{t}}$
by using the same graphical construction with mutations treated as
deaths.
\item [{$^{2}{{\zeta}_{t}}$,}] a one-type process coupled to $^{2}{{\xi}_{t}}$
by using the same graphical construction and ignoring mutations.
\item [{${{\delta}_{*}},{{\delta}^{*}}$,}] the weak, and strong survival
critical values of the contact process, respectively.
\item [{$\alpha=P\left(\zeta_{t}^{x}\ne\varnothing\,\,\forall t>0\right)$,}] the
probability of survival of the contact process. 
\item [{${{\eta}_{t}}$,}] the two-type contact process (on $S=\mathbb{\ensuremath{Z}}$).
\end{description}

\subsection*{Notation specific to the proof of Theorem 2.}
\begin{description}
\item [{$\hat{\xi}_{t}^{\left(x,s\right)}$,}] the backward time ancestral
process of $x$ at time $s$, $t$ time units in the past. (the list
of possible ancestors of $\left(x,s\right)$ at time $s-t$, in order
of primality).
\item [{$\overline{\xi}_{s}^{x}$,}] the forward time ancestral process,
defined for all $s\ge0$, with the same law as $\hat{\xi}_{s}^{\left(x,t\right)}$
for all $s\le t$. 
\item [{$\overline{\xi}_{t}^{\left(x,s\right)}$,}] the forward time ancestral
process of $\left(x,s\right)$ ($x$ at time $s$), evaluated at time
$t$. 
\item [{$A_{t}^{x}$,}] a list of first priority possible ancestors (which
can only propagate into type 2) of $\left(x,t\right)$.
\item [{$OC\left(\hat{\xi}_{t}^{\left(x,s\right)},\eta\right)$,}] the
event that the first occupied vertex of $\eta$ in the ancestor list
$\hat{\xi}_{t}^{\left(x,s\right)}$ is 1-viable and is actually occupied
by a 1 in the configuration $\eta$. 
\end{description}

\section*{Acknowledgment}

G.A. was supported by the Israel Science Foundation grant \#575/16
and the German Israeli Foundation grant \#I-1363-304.6/2016. 

\bibliographystyle{plain}

\begin{thebibliography}{1}
	
	\bibitem{Andjel2010}
	Enrique~D Andjel, Judith~R Miller, Etienne Pardoux, et~al.
	\newblock Survival of a single mutant in one dimension.
	\newblock {\em Electronic Journal of Probability}, 15:386--408, 2010.
	
	\bibitem{Cox2009}
	J~Theodore Cox, Rinaldo~B Schinazi, et~al.
	\newblock Survival and coexistence for a multitype contact process.
	\newblock {\em The Annals of Probability}, 37(3):853--876, 2009.
	
	\bibitem{durrett2009coexistence}
	Rick Durrett.
	\newblock Coexistence in stochastic spatial models.
	\newblock {\em The Annals of Applied Probability}, pages 477--496, 2009.
	
	\bibitem{Harris1974}
	Theodore~E Harris.
	\newblock Contact interactions on a lattice.
	\newblock {\em The Annals of Probability}, pages 969--988, 1974.
	
	\bibitem{liggett2013stochastic}
	Thomas~M Liggett.
	\newblock {\em Stochastic interacting systems: contact, voter and exclusion
		processes}, volume 324.
	\newblock springer science \& Business Media, 2013.
	
	\bibitem{Neuhauser1992}
	Claudia Neuhauser.
	\newblock Ergodic theorems for the multitype contact process.
	\newblock {\em Probability Theory and Related Fields}, 91(3):467--506, 1992.
	
\end{thebibliography}

\end{document}